\documentclass[12pt]{article}
\usepackage{a4wide}

\usepackage{amsmath,amssymb,amsthm}
\usepackage[mathscr]{eucal}
\usepackage{tikz}
\usetikzlibrary{positioning,automata}
\usetikzlibrary{arrows,calc}

\title{\textbf{The lamplighter group $\mathbb{Z}_3\wr\mathbb{Z}$ generated by a bireversible automaton}}

\author{Ievgen Bondarenko\\
\and
Daniele D'Angeli\\
\and
Emanuele Rodaro \\
}

\newcommand{\Addresses}{{
\bigskip
\footnotesize
\noindent I.~Bondarenko, \emph{Department of Algebra and Mathematical Logic Mechanics and Mathematics Faculty 
National Taras Shevchenko University of Kyiv.}
\textit{E-mail address}:\texttt{ievgbond@gmail.com}

\medskip

\noindent D.~D'Angeli, \emph{Institut f\"{u}r Mathematische Strukturtheorie (Math C), Technische Universit\"{a}t Graz, Steyrergasse 30, 8010 Graz.}
\textit{E-mail address}:\texttt{dangeli@math.tugraz.at}

\medskip
  
\noindent E.~Rodaro, \emph{Centro de Matem\'atica, University of Porto, Rua do Campo Alegre, 687, Porto, 4169-007, Portugal.}
\textit{E-mail address}:\texttt{emanuele.rodaro@fc.up.pt}
}}

\newcommand{\dual}{\partial}
\newcommand{\inv}{\textit{i}}
\newcommand{\Aut}{\textrm{Aut}}
\newcommand{\Sym}{\textrm{Sym}}

\newtheorem{theorem}{Theorem}
\newtheorem{proposition}[theorem]{Proposition}
\newtheorem{corollary}{Corollary}[theorem]
\newtheorem{lemma}{Lemma}
\theoremstyle{definition}

\newtheorem{remark}{Remark}

\begin{document}

\date{}

\maketitle

\begin{abstract}
We construct a bireversible self-dual automaton with $3$ states over an alphabet with
$3$ letters which generates the lamplighter group $\mathbb{Z}_3\wr\mathbb{Z}$.

\vspace{0.2cm}\textit{2010 Mathematics Subject Classification}: 20F65, 20M35

\end{abstract}


\section{Introduction}

Groups generated by automata (or automaton groups) are interesting from several points
of view. First of all, automaton groups provide simple examples of groups with many
extraordinary properties: finitely generated infinite torsion groups, groups of
intermediate growth, just-infinite groups, groups with non-uniformly exponential
growth. At the same time, groups generated by automata arise in various areas of
mathematics: in fractal geometry via limit spaces of automaton groups, in complex
dynamics via iterated monodromy groups, in graph theory via Schreier graphs of
automaton groups, in dynamical systems via limit dynamical systems of automaton groups,
in game theory via algebraic models of games (see
\cite{BarthSilva,gri_sunik:branching,self_sim_groups} and the reference therein).



There is an ongoing project to understand which groups can be realized by finite
automata. We mention only a few results in this direction relative to the current
paper. In \cite{GriZuk:lampl} Grigorchuk and \.{Z}uk showed that the lamplighter group
$\mathbb{Z}_2\wr\mathbb{Z}$ can be generated by a $2$-state automaton over a $2$-letter
alphabet, which further lead to a negative answer to the strong Atiyah conjecture
concerning $L^2$-Betti numbers \cite{Ariyah}. Silva and Steinberg
\cite{SilvaStein:lamp} realized the lamplighter groups $\mathbb{Z}_n\wr\mathbb{Z}$ by
the so-called reset automata, these automata were further generalized in
\cite{groups_coloring}. Some solvable automaton groups were realized by Bartholdi and
\v{S}uni\'{k} in \cite{BartSun:solvable}.


There are two standard operations that can be performed on automata: taking dual automaton by interchanging
the alphabet with the set of states, and taking inverse automaton by switching input
and output letters (in general, the inverse of an automaton may be not well-defined,
i.e., it may be not a (deterministic) automaton). By applying these two operations to
any automaton one can produce up to eight automata. If all these automata are
well-defined, the original automaton is called bireversible. The study of bireversible
automata was initiated by Macedo\'{n}ska, Nekrashevych and Sushchansky
\cite{MacNekrSush} in connection with the commensurator of a regular (unrooted) tree.
New geometric ideas came to the area with the paper \cite{square_compl} of Glasner and
Mozes, who associated a square complex to each finite automaton and noticed that an
automaton is bireversible if and only if the universal covering of the associated
square complex is a topological product of two trees. This approach led to the first
automaton realizations of free groups.

It is surprisingly difficult to describe a group generated by a bireversible automaton.
For example, there are only two bireversible automata with $3$ states over an alphabet
with $2$ letters generating infinite groups --- Aleshin and Bellaterra automata. The
Bellaterra automaton generates the free product of three copies of $C_2$
\cite[Theorem~1.10.2]{self_sim_groups}. As concerning Aleshin automaton, it was an open
question for a long time whether this automaton generates the free group of rank three,
until this was confirmed by Vorobets and Vorobets \cite{Vorobets:Aleshyn}. Two families
of bireversible automata generalizing Aleshin and Bellaterra automata were studied in
\cite{Vorobets:Series,SteinVor,SavVor}: the automata in these families generate a free
group of finite rank or the free product of copies of $C_2$. Up to now all investigated
bireversible automata generate finitely presented groups, while all the automata
generating $(\mathbb{Z}_n)^{k}\wr\mathbb{Z}$, $k\ge 1$, from \cite{BartSun:solvable,
groups_coloring,GriZuk:lampl,SilvaStein:lamp} are not bireversible as well as all known
automata generating infinite torsion groups and groups of intermediate growth.

In this paper we consider the automaton $A$ with $3$ states over an alphabet with
$3$ letters shown in Figure~\ref{Fig_Automaton}. This automaton is bireversible and its
dual automaton $\dual(A)$ is equivalent to $A$: the correspondence $a\mapsto 1$,
$b\mapsto 3$, $c\mapsto 2$ converts $A$ to $\dual(A)$. Therefore all eight automata
obtained from $A$ by taking dual and inverse automata generate isomorphic groups. Our
goal is to prove the following main theorem.

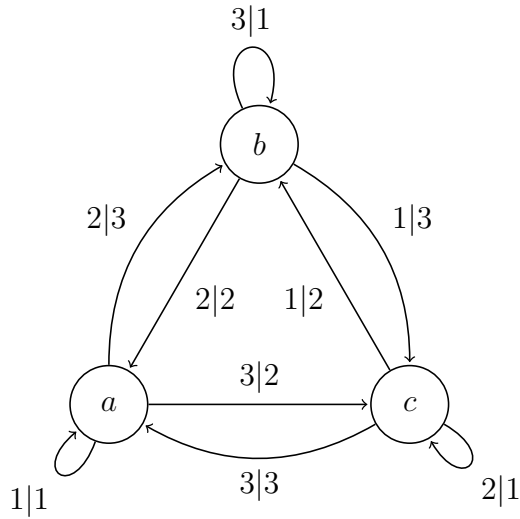
\begin{figure}
\begin{center}
\begin{tikzpicture}[shorten >=1pt,node distance=1cm, on grid,auto,/tikz/initial text=,semithick,transform shape]
  \node[state] at (0,0) (a) {$a$};
  \node[state] at (2,3.46) (b) {$b$};
  \node[state] at (4,0) (c) {$c$};
    \path[->]
    (a) edge [in=220, out=250, loop] node  {$1|1$} (a)
    (b) edge [in=75, out=115, loop] node [above]  {$3|1$} (b)
    (c) edge [in=-60, out=-30, loop] node  {$2|1$} (c)
    (a) edge [bend left] node {$2|3$} (b)
    (b) edge [bend left] node {$1|3$} (c)
    (c) edge [bend left] node {$3|3$} (a)
    (a) edge  node {$3|2$} (c)
    (c) edge  node {$1|2$} (b)
    (b) edge  node {$2|2$} (a);
\end{tikzpicture}
\caption{The automaton $A$ generating the lamplighter group $\mathbb{Z}_3\wr\mathbb{Z}$}\label{Fig_Automaton}
\end{center}
\end{figure}

\begin{theorem}\label{thm_main}
The group $G_A$ generated by the automaton $A$ is isomorphic to the lamplighter group
$\mathbb{Z}_3\wr\mathbb{Z}$.
\end{theorem}

In particular, we get an example of a bireversible automaton generating infinitely
presented group. While finishing this article, we were informed by D.~Savchuk and S.~Sidki
that they proved that a certain bireversible automaton with $4$ states over an alphabet
with $2$ letters generates the infinitely presented group
$\left((\mathbb{Z}_2\times\mathbb{Z}_2)\wr \mathbb{Z}\right)\rtimes \mathbb{Z}_2$.

\section{Preliminaries}

Let $X$ be a finite alphabet and $X^{*}$ the free monoid freely generated by $X$. The
elements of $X^{*}$ are finite words $v=x_1x_2\ldots x_n$, $x_i\in X$,
$n\in\mathbb{N}$, together with the empty word denoted $\emptyset$. The operation is
concatenation of words. The length of a word $v=x_1x_2\ldots x_n$ is $|v|=n$. We will
also consider the space $X^{\mathbb{N}}=\{x_1x_2\ldots: \ x_i\in X\}$ of all right
infinite sequences over $X$ with the product topology of discrete sets $X$.

An automaton $A$ over the alphabet $X$ is a finite directed labeled graph, whose
vertices are called the states of the automaton, and where each edge is labeled by a
pair $x|y$ for some letters $x,y\in X$ in such a way that for each vertex $s\in A$ and
every letter $x\in X$ there exists exactly one arrow outgoing from $s$ and labeled by
$x|y$ for some $y\in X$. Such automata are precisely finite complete deterministic
Mealy automata with the same input and output alphabets.

The dual automaton $\dual(A)$ is obtained by interchanging the alphabet with the set of
states: the states of $\dual(A)$ are the elements of $X$ and the arrows are given by
the rule
\[
x\xrightarrow{s|t}y \mbox{ \ in $\dual(A)$ \quad if \quad }  s\xrightarrow{x|y}t \mbox{ \ in $A$}.
\]
The dual automaton $\dual(A)$ is always well-defined. The inverse automaton $\inv(A)$
is obtained by switching labels of arrows: the states of $\inv(A)$ are formal symbols
$s^{-1}$ for $s\in A$ and the arrows are given by
\[
s^{-1}\xrightarrow{y|x}t^{-1} \mbox{ \ in $\inv(A)$ \quad if \quad }  s\xrightarrow{x|y}t \mbox{ \ in $A$}.
\]
The $\inv(A)$ is not always an automaton: there may be several arrows with the same
left label outgoing from the same vertex. If $\inv(A)$ is an automaton then $A$ is
called invertible. An automaton $A$ is called bireversible if all eight automata are
well-defined:
\[
A, \  \dual(A), \ \inv(A), \ \inv(\dual(A)), \ \dual(\inv(A)), \ \dual(\inv(\dual(A))), \ \inv(\dual(\inv(A))), \ \inv(\dual(\inv(\dual(A))))=\dual(\inv(\dual(\inv(A)))).
\]
It is easy to see that $A$ is bireversible if $A$, $\dual(A)$ and $\dual(\inv(A))$ are
invertible.

Let us describe how to generate groups by automata. Every state $s$ of an automaton $A$
defines the transformation $s:X^{*}\rightarrow X^{*}$ as follows. Given a word
$v=x_1x_2\ldots x_n\in X^{*}$, there exists a unique directed path in the automaton $A$
starting at the state $s$ and labeled by $x_1|y_1$, $x_2|y_2$,\ldots,$x_n|y_n$ for some
$y_i\in X$. Then the word $y_1y_2\ldots y_n$ is called the image of $x_1x_2\ldots x_n$
under $s$, and the vertex at the end of this path is called the section of $s$ at $v$
denoted $s|_v$. The action of an automaton $A$ on the space $X^{*}$ naturally
extends to the action on the space $X^{\mathbb{N}}$. All the transformations given by
the states of $A$ are invertible if and only if $A$ is invertible; in this case these
transformations generate a group under composition of functions called the automaton
group $G_A$ generated by the automaton $A$ (we will be using left actions).

A convenient way to work with transformations given by the states of an automaton are
wreath recursions. Every invertible automaton with states $\{s_1,\ldots,s_m\}$ over the
alphabet $X=\{1,2,\ldots,d\}$ can be uniquely given by the following system called
wreath recursion:
\begin{align}\label{eqn_wreath_recursion}
s_1&=(s_{11},s_{12},\ldots,s_{1d})\pi_1,\nonumber\\
s_2&=(s_{21},s_{22},\ldots,s_{2d})\pi_2,\\
  & \vdots \nonumber\\
s_m&=(s_{m1},s_{m2},\ldots,s_{md})\pi_m,\nonumber
\end{align}
where $s_{ij}=s_i|_j\in \{s_1,\ldots,s_m\}$ is the section of $s_i$ at $j$, and
$\pi_i\in Sym(X)$ is the permutation induced by the action of $s_i$ on $X$. The tuples
$(s_{ij})$ describe the arrows in automaton while permutations $\pi_i$ describe the
labels of arrows: we have an arrow from the vertex $s_i$ to vertex $s_{ij}$ labeled by $j|\pi_i(j)$.
The system (\ref{eqn_wreath_recursion}) defines the action of each
state $s_i$ on words over $X$ by the recursive rule:
\[
s_i(xv)=\pi_i(x)s_{ix}(v) \mbox{ \ for $x\in X, v\in X^{*}\cup X^{\mathbb{N}}$ and $i=1,\ldots,m$.}
\]

Similarly, one can use wreath recursions to work with elements of automaton groups.
Define the section of a product (or word) $s_1s_2\ldots s_n$, $s_i\in A^{\pm 1}$ at
$v\in X^{*}$ by the rule
\begin{equation*}
(s_1s_2\ldots s_n)|_v=s'_1s'_2\ldots s'_n, \ \mbox{where } s'_i=s_i|_{(s_1\ldots s_{i-1})(v)}.
\end{equation*}
Then every element $g\in G_A$ over the alphabet $X=\{1,2,\ldots,d\}$ can be decomposed
as
\begin{equation}\label{eqn_wreath_decomp}
g=(g|_1,g|_2,\ldots,g|_d)\pi_g,
\end{equation}
where $\pi_g\in Sym(X)$ is the permutation induced by the action of $g$ on $X$, and
$g|_i$ are the sections of $g$ at elements of $X$. The inverse and multiplication of
elements written in this form can be performed by the rules
\begin{align*}
g^{-1}&=(g|_{\pi_g^{-1}(1)},g|_{\pi_g^{-1}(2)},\ldots,g|_{\pi_g^{-1}(d)})\pi_g^{-1},\\
g\cdot h&=(g|_1h|_{\pi_g(1)},g|_2h|_{\pi_g(2)},\ldots,g|_dh|_{\pi_g(d)})\pi_g\pi_h.
\end{align*}

There is a direct connection between sections of words over states of $A$ and the action of the dual automaton $\partial(A)$. Elements of the group generated by $\partial(A)$ act on words over the states of $A$.
If $v$ is a word over alphabet (an element of $G_{\partial(A)}$) and $w$ is a word over states, then the image $v(w)$ is equal to the section $w|_v$.

The terminology of wreath recursions comes from the wreath decomposition of
automorphism groups of regular rooted trees. The set $X^{*}$ can be identified with the
vertex set of a rooted tree with empty word as the root and with edges $(v,vx)$ for
$x\in X$, $v\in X^{*}$. For this reason words over $X$ are usually called vertices.
The set $X^n$ of words of lengths $n$ is called the $n$-th level of the tree $X^{*}$.
The set $X^{\mathbb{N}}$ can be identified with the boundary of $X^{*}$. The
transformations defined by invertible automata over $X$ act by automorphisms on the
tree $X^{*}$ and by homeomorphisms on the space $X^{\mathbb{N}}$. The automorphism
group $\Aut(X^{*})$ can be decomposed as the permutational wreath product
$\Aut(X^{*})\cong \Aut(X^{*})\wr \Sym(X)$, that explains why we have the
decomposition~(\ref{eqn_wreath_decomp}). The permutation $\pi_g$ from
(\ref{eqn_wreath_decomp}) is called the root permutation of $g$. Note that any element
$g$ can be uniquely given by the collection $(\pi_{g|_v})_{v\in X^{*}}$ of root
permutations of all sections of $g$.

\section{Proof of Theorem~\ref{thm_main}}

Let $A$ be the automaton shown in Figure~\ref{Fig_Automaton} with the set of states
$S=\{a,b,c\}$ over the alphabet $X=\{1,2,3\}$, and let $G_A$ be the group generated by
$A$. The automaton $A$ and its dual $\dual(A)$ can be given by the following wreath
recursions:
\begin{align*}
a&=(a,b,c)(2,3), & 1&=(1,3,2)(b,c),\\
b&=(c,a,b)(1,3), & 2&=(3,2,1)(a,b),\\
c&=(b,c,a)(1,2), & 3&=(2,1,3)(a,c).
\end{align*}
So that \begin{align*}
a^{-1}&=(a^{-1},c^{-1},b^{-1})(2,3), & 1^{-1}&=(1^{-1},2^{-1},3^{-1})(b,c),\\
b^{-1}&=(b^{-1},a^{-1},c^{-1})(1,3), & 2^{-1}&=(2^{-1},3^{-1},1^{-1})(a,b),\\
c^{-1}&=(c^{-1},b^{-1},a^{-1})(1,2), & 3^{-1}&=(3^{-1},1^{-1},2^{-1})(a,c).
\end{align*}

The wreath recursions for $ab^{-1},bc^{-1},ca^{-1}$:
\begin{align*}
ab^{-1}&=(ab^{-1}, bc^{-1}, ca^{-1})(1,3,2),\\
bc^{-1}&=(ca^{-1}, ab^{-1}, bc^{-1})(1,3,2),\\
ca^{-1}&=(bc^{-1}, ca^{-1}, ab^{-1})(1,3,2),
\end{align*}
imply the relations $ab^{-1}=bc^{-1}=ca^{-1}$. We denote $\alpha=ab^{-1}$. Then
$\alpha$ has order $3$ and satisfies the wreath recursion
\[
\alpha=(\alpha,\alpha,\alpha)(1,3,2).
\]
Analogously one gets relations
\[
\alpha^{-1}=ac^{-1}=ba^{-1}=cb^{-1}, \quad a^{-1}b=b^{-1}c=c^{-1}a, \quad a^{-1}c=b^{-1}a=c^{-1}b.
\]
Now it is clear that the group $G_A$ is generated by $a$ and $\alpha$. Our goal is to
prove that $G_A$ has the following presentation:
\begin{equation}\label{eqn_presentation}
G_A=\langle a,\alpha | \alpha^3, [a^{-n}\alpha a^n,a^{-m}\alpha a^m], n,m\in\mathbb{Z} \rangle\cong \mathbb{Z}_3\wr\mathbb{Z}.
\end{equation}

Let us define the subgroup $W$ of $\Aut(X^{*})$ consisting of elements $g$ such that
the root permutations of all sections of $g$ belong to $Alt_3=\{\varepsilon, (1,2,3),
(1,3,2)\}$ and for each $n\in\mathbb{N}$ the root permutations of $g|_v$ at all
vertices $v$ of $n$-th level are equal. In other words, each element $g$ of $W$ can be
given by a sequence $(\pi_1,\pi_2,\ldots)$, $\pi_i\in Alt_3$, where $\pi_i$ is the root
permutation of $g|_v$ for all vertices $v\in X^{i-1}$; such element acts on $X^{*}$ as
follows:
\[
g(x_1x_2\ldots x_n)=\pi_1(x_1)\pi_2(x_2)\ldots \pi_n(x_n)
\]
for any $x_i\in X$ and $n\in\mathbb{N}$. It is easy to see that the group $W$ is
abelian of exponent $3$. Notice that the elements $\alpha$, $\alpha^{-1}$, $a^{-1}b$,
$a^{-1}c$ belong to $W$ and that every $g\in W$ can be decomposed as $g=(h,h,h)\pi$ for
some $h\in W$ and $\pi\in Alt_3$.

\begin{lemma}\label{Lemma_about_W}
For any $g\in W$ and $x,y\in\{a,b,c\}$ the elements $x^{-1}gy$ and $xgy^{-1}$ belong to
$W$.
\end{lemma}
\begin{proof}
For every $g\in W$, we split the elements $x^{-1}gy$ on three types:
\begin{eqnarray*}
T_1(g)=\{a^{-1}ga,b^{-1}gb,c^{-1}gc\},\\ T_2(g)=\{a^{-1}gb,b^{-1}gc,c^{-1}ga\},\\ T_3(g)=\{a^{-1}gc,b^{-1}ga,c^{-1}gb\}.
\end{eqnarray*}
We will prove that for each $n\in\mathbb{N}$ the set of sections of $x^{-1}gy$ at
vertices of $n$-th level is equal to the set $T_i(g_n)$ for some $i\in\{1,2,3\}$
depending on $n$, where $g_n$ is the section of $g$ at some vertex of $n$-th level
(they are all equal), i.e.,
\[
\{(x^{-1}gy)|_{v} : v\in X^n\}=T_i(g_n).
\]
For $n=1$, let us write $g=(h,h,h)\pi$, $h\in W$ and note that $\{ (x^{-1}gy)|_v : v\in
X \}$ is equal to $T_i(h)$ for some $i\in\{1,2,3\}$. Assume inductively that the claim
holds for level $n$. By direct computations shown in Table~\ref{Table_Computations} we
see that for every $i\in\{1,2,3\}$
\[
\{ f|_v : f\in T_i(g) \mbox{ and } v\in X  \}=T_j(h)
\]
for some $j\in\{1,2,3\}$ depending only on $i$ and $\pi$. Therefore
\[
\{(x^{-1}gy)|_{v} : v\in X^{n+1}\}=\{ f|_v : f\in T_i(g_n) \mbox{ and } v\in X \}= T_j(g_{n+1}),
\]
and the claim is proved.

Notice that for each $i\in\{1,2,3\}$ and any $g\in W$ all elements of $T_i(g)$ have the
same root permutations which belong to $Alt_3$. Hence for each $n\in\mathbb{N}$ the
root permutations of $(x^{-1}gy)|_v$ at the vertices $v$ of $n$-th level are all equal.
Therefore $x^{-1}gy\in W$ for any $x,y\in\{a,b,c\}$.

Analogously one can show that $xgy^{-1}\in W$ for any $x,y\in\{a,b,c\}$.
\end{proof}

\begin{remark}
It follows from the lemma that actually all elements in each $T_i(g)$ are equal, because they have the same set of permutation on each level:
\[
a^{-1}ga=b^{-1}gb=c^{-1}gc, \qquad a^{-1}gb=b^{-1}gc=c^{-1}ga, \qquad a^{-1}gc=b^{-1}ga=c^{-1}gb.
\]
Therefore $T_1,T_2,T_3$ can be considered as transformations of $W$. Since every
element in $W$ can be uniquely represented by a sequence of permutations
$(\pi_1,\pi_2,\ldots)$, $\pi_{i}\in Alt_3$, i.e., an element from $Alt_3^{\mathbb{N}}$,
the $T_1,T_2,T_3$ can be viewed as transformations of $Alt_3^{\mathbb{N}}$.
Table~\ref{Table_Computations} shows that these transformations satisfy the wreath
recursion
\begin{align*}
T_1&=(T_1,T_2,T_3)( \tau_2, \tau_3 )\\
T_2&=(T_3,T_1,T_2)( \tau_1, \tau_3), \\
T_3&=(T_2,T_3,T_1)( \tau_1, \tau_3)
\end{align*}
where $\tau_1=\varepsilon, \tau_2=(1,2,3), \tau_3=(1,3,2)$. Interestingly, this
recursion repeats the wreath recursion for the automaton $A$ if we identify $\tau_i$
with $i$.

We can apply this observation to elements of $G_A$ of the form $u^{-1}v$ for
$u,v\in\{a,b,c\}^n$. In this case $u^{-1}v$ corresponds to a word $w$ in the alphabet
$\{T_1, T_2, T_3\}$ of length $n$, because we can write $u^{-1}v$ as the progressive
composition of some conjugations of the form $x^{-1}y$. For example:
\[
a^{-1}b^{-1}b^{-1}c^{-1}abaa=T_2T_1T_3T_1(e).
\]
The trivial element $e\in W$ is represented by the sequence
$\tau_1^{\infty}=\tau_1\tau_1\ldots$. Therefore the sequence of permutations
representing $u^{-1}v$ is equal to the image $w(\tau_1^{\infty})$.
\end{remark}


For a word $w$ over $\{a^{\pm 1},b^{\pm 1},c^{\pm 1}\}$ we denote by $ord(w)$ the sum
of exponents of letters in $w$. For example, $ord(a^{-1}b^3cb^{-2})=-1+3+1-2=1$.

\begin{lemma}\label{Lemma_ord(w)=0}
If an element $g\in G_{A}$ can be represented by a word over $\{a^{\pm 1},b^{\pm 1},c^{\pm
1}\}$ with $ord(w)=0$, then $g\in W$. In particular, $g^3=e$ and such elements commute
with each other.
\end{lemma}
\begin{proof}

We prove the lemma by induction on the length of $g$. The statement holds for elements
$x^{-1}y$ and $xy^{-1}$ for all $x,y\in\{a,b,c\}$. If we assume that the statement
holds for elements $g$ of length $\leq 2n$, then it holds for elements $x^{-1}gy$ and
$xgy^{-1}$ by Lemma~\ref{Lemma_about_W}. So let us prove the statement for words
starting and ending with either both letters in $X$ or both letters in $X^{-1}$. These
words are either of type $xvy$ with $ord(v)=-2$ or of type $x^{-1}vy^{-1}$ with
$ord(v)=2$. It is easy to show by induction, that if $|w|=2(n+1)$, $w=xvy$ with
$ord(v)=-2$ (resp. $w=x^{-1}vy^{-1}$ with $ord(v)=2$) then $w$ is a concatenation of
words of the form $xuy^{-1}$ and $x^{-1}uy$ (resp. $x^{-1}uy$ and $xuy^{-1}$) for some
$u$ with $ord(u)=0$ and length $\leq 2n$. The statement follows.
\end{proof}

Notice that all elements $a^{-n}\alpha a^n$ from presentation (\ref{eqn_presentation})
satisfy the condition of Lemma~\ref{Lemma_ord(w)=0}.

\begin{table}[!t]
 \centering
\begin{tabular}{|c|l|}
\hline
& \hspace{2.5cm} $g=(h,h,h)\pi$\\ \hline
$\pi=\varepsilon$ &
$\begin{array}{cl}
                 &\vspace{-0.3cm}\\
& a^{-1}ga=(a^{-1}ha,c^{-1}hc,b^{-1}hb)\\
\mbox{Type I}: & b^{-1}gb=(b^{-1}hb,a^{-1}ha,c^{-1}hc)\\
& c^{-1}gc=(c^{-1}hc,b^{-1}hb,a^{-1}ha)\\
& \\
& a^{-1}gb=(a^{-1}hc,c^{-1}hb,b^{-1}ha)(1,3,2)\\
\mbox{Type II}: & b^{-1}gc=(b^{-1}ha,a^{-1}hc,c^{-1}hb)(1,3,2)\\
& c^{-1}ga=(c^{-1}hb,b^{-1}ha,a^{-1}hc)(1,3,2)\\
& \\
& a^{-1}gc=(a^{-1}hb,c^{-1}ha,b^{-1}hc)(1,2,3)\\
\mbox{Type III}: & b^{-1}ga=(b^{-1}hc,a^{-1}hb,c^{-1}ha)(1,2,3)\\
& c^{-1}gb=(c^{-1}ha,b^{-1}hc,a^{-1}hb)(1,2,3)\\
&\vspace{-0.3cm}
\end{array}$\\ \hline
$\pi=(1,2,3)$ &
$\begin{array}{cl}
                 &\vspace{-0.3cm}\\
& a^{-1}ga=(a^{-1}hb,c^{-1}ha,b^{-1}hc)(1,3,2)\\
\mbox{Type I}: & b^{-1}gb=(b^{-1}hc,a^{-1}hb,c^{-1}ha)(1,3,2)\\
& c^{-1}gc=(c^{-1}ha,b^{-1}hc,a^{-1}hb)(1,3,2)\\
& \\
& a^{-1}gb=(a^{-1}ha,c^{-1}hc,b^{-1}hb)(1,2,3)\\
\mbox{Type II}: & b^{-1}gc=(b^{-1}hb,a^{-1}ha,c^{-1}hc)(1,2,3)\\
& c^{-1}ga=(c^{-1}hc,b^{-1}hb,a^{-1}ha)(1,2,3)\\
& \\
& a^{-1}gc=(a^{-1}hc,c^{-1}hb,b^{-1}ha)\\
\mbox{Type III}: & b^{-1}ga=(b^{-1}ha,a^{-1}hc,c^{-1}hb)\\
& c^{-1}gb=(c^{-1}hb,b^{-1}ha,a^{-1}hc) \\
&\vspace{-0.3cm}
\end{array}$\\ \hline
$\pi=(1,3,2)$ &
$\begin{array}{cl}
                 &\vspace{-0.3cm}\\
& a^{-1}ga=(a^{-1}hc,c^{-1}hb,b^{-1}ha)(1,2,3)\\
\mbox{Type I}: & b^{-1}gb=(b^{-1}ha,a^{-1}hc,c^{-1}hb)(1,2,3)\\
& c^{-1}gc=(c^{-1}hb,b^{-1}ha,a^{-1}hc)(1,2,3)\\
& \\
& a^{-1}gb=(a^{-1}hb,c^{-1}ha,b^{-1}hc)\\
\mbox{Type II}: & b^{-1}gc=(b^{-1}hc,a^{-1}hb,c^{-1}ha)\\
& c^{-1}ga=(c^{-1}ha,b^{-1}hc,a^{-1}hb)\\
& \\
& a^{-1}gc=(a^{-1}ha,c^{-1}hc,b^{-1}hb)(1,3,2)\\
\mbox{Type III}: & b^{-1}ga=(b^{-1}hb,a^{-1}ha,c^{-1}hc)(1,3,2)\\
& c^{-1}gb=(c^{-1}hc,b^{-1}hb,a^{-1}ha)(1,3,2)\\
&\vspace{-0.3cm}
\end{array}$\\ \hline
\end{tabular}
\caption{Decomposition of $x^{-1}gy$ for $x,y\in\{a,b,c\}$ and $g\in
W$}\label{Table_Computations}
\end{table}

\begin{corollary}
The relations $[a^{-n}\alpha a^n,a^{-m}\alpha a^m]=e$ hold in the group $G_A$.
\end{corollary}

\begin{lemma}
The group $G_A$ acts transitively on $X^n$ for every $n\in\mathbb{N}$.
\end{lemma}
\begin{proof}
We will prove that the stabilizer $St_{G_{A}}(1^n)$ of the vertex $1^n$ acts transitively on
$1^nX$ for each $n\in\mathbb{N}$. The statement immediately follows from this claim.
Indeed, let $u,v\in X^n$ and we want to find an element $g\in G_A$ such that $g(u)=v$.
We proceed by induction on the length $|u|=|v|=n$. If $n=1$ the claim follows from the
transitivity of $G_A$ on $X$. Let $|u|=|v|=n+1>1$ with $u=u'x$ and $v=v'y$, where
$u',v'\in X^n$ and $x,y\in X$. By induction there exist $h,k\in G_A$ such that
$h(u')=k(v')=1^n$. Let $x'=h|_{u'}(x)$ and $y'=k|_{v'}(y)$. Since $St_{G_{A}}(1^n)$ is
transitive on $1^nX$, there is $s\in St_{G_{A}}(1^n)$ such that $s(1^nx')=1^ny'$. Then if we
put $g=hsk^{-1}$, we have $g(u)=v$.

Let us prove the claim. Notice that $a^3,b^3,c^3$ preserve the vertex set
$X11=\{x_1=111,y_1=211,z_1=311\}$, and the action restricted to $\{x_1,y_1,z_1\}$ can
be described by the wreath recursion
\begin{align*}
a^3&=(a^3,b^3,c^3)(y_1,z_1),\\
b^3&=(c^3,a^3,b^3)(x_1,z_1),\\
c^3&=(b^3,c^3,a^3)(x_1,y_1).
\end{align*}
which repeat the wreath recursion for the automaton $A$. It immediately follows by
induction that if we denote
\begin{eqnarray*}
x_k=1^{3^k}, y_k=21^{3^{k-1}+2}, z_k=31^{3^{k-1}+2}\quad \mbox{ and }\quad
a_k=a^{3^k}, b_k=b^{3^k}, c_k=c^{3^k},
\end{eqnarray*}
then $a_k,b_k,c_k$ preserve the vertex set $\{x_k,y_k,z_k\}$ with the recursion
\begin{align*}
a_k&=(a_k,b_k,c_k)(y_k,z_k),\\
b_k&=(c_k,a_k,b_k)(x_k,z_k),\\
c_k&=(b_k,c_k,a_k)(x_k,y_k).
\end{align*}
It follows that
\[
b_k^2(x_k)=x_k \quad \mbox{ and } \quad b_k^2|_{x_k}=c_kb_k=c^{3^k}b^{3^k}.
\]
The element $c^{3^k}b^{3^k}$ acts as permutation $(1,2,3)$ on $X$. Therefore $b_k^2$
stabilizes $x_k$ and acts transitively on the set $x_kX$. This means that the
stabilizer of $1^n$ acts transitively on $1^nX$ for each $n=3^k$, $k\in\mathbb{N}$. By
taking sections at vertices $1^i$ we get that $St_{G_{A}}(1^{3^{k}-i})$ is transitive on $1^{3^{k}-i}X$ which implies our claim for each $n\in\mathbb{N}$.
\end{proof}

\begin{corollary}\label{Cor_sections}
Any word of length $n$ over $\{a,b,c\}$ is a section of any other word of length $n$
over $\{a,b,c\}$.
\end{corollary}
\begin{proof}
Since the automaton $A$ is equivalent to its dual automaton $\dual(A)$, the group
generated by $\dual(A)$ acts transitively on $\{a,b,c\}^n$ for each $n\in\mathbb{N}$.
Therefore the semigroup generated by $X$ acts transitively on $\{a,b,c\}^n$. This means
that for any $w,w'\in\{a,b,c\}^n$ there exists $v\in X^{*}$ such that $w|_v=w'$.
\end{proof}

\begin{corollary}\label{Cor_free_semigroup}
The semigroup generated by $a,b,c$ is free.
\end{corollary}
\begin{proof}
Every non-empty word $w$ over $\{a,b,c\}$ is a non-trivial element, because $w$ has a
non-trivial section $a^nb$, $n=|w|-1$.

Assume by contradiction that $w=_{G_{A}}v$ for a word $v$ over $\{a,b,c\}$ different from
$w$. If $|v|>|w|$ then $v_1^{-1}w=_G v_2$, where $v=v_1v_2$ with $|v_1|=|w|$ and non-empty $v_2$. Since $ord(v_1^{-1}w)=0$ we get $v_2^3=e$ by Lemma \ref{Lemma_ord(w)=0} contradicting our first statement.

We may assume that $|v|=|w|$, and let us take a shortest such relation $v=_{G_{A}}w$. By Corollary \ref{Cor_sections} there is $s\in X^{*}$ such that $v|_{s}=a^n$. Let $u:=w|_{s}$. Since $A$ is bireversible, $a^n\neq u$. Furthermore $a^n=_{G_{A}}u$ is a relation where the first letter of $u$ is $b$ or $c$, because of minimality.
Taking the section at $11$ of this relation, we get $a^n=a^n|_{11}=_{G_{A}}u|_{11}=_{G_{A}}u$. The first letters of
$u$ and $u|_{11}$ are equal since $b|_{11}=b$ and $c|_{11}=c$. After canceling these letters we get a shorter relation.
Therefore $u$ and $u|_{11}$ are equal as words. However, this is impossible already for
the first two letters of $w$ as the following computations show:
\begin{align*}
ba|_{11}=bc, \quad bb|_{11}=ba, \quad bc|_{11}=bb,\\
ca|_{11}=cb, \quad cb|_{11}=cc, \quad cc|_{11}=ba.
\end{align*}
\end{proof}

We are ready to prove Theorem~\ref{thm_main}.

\begin{proof}[Proof of Theorem~\ref{thm_main}]
Let $N$ be the subgroup of $G_A$ generated by elements $a^{-n}\alpha a^n$,
$n\in\mathbb{Z}$. Then $N$ is a normal abelian subgroup of exponent $3$ by
Lemma~\ref{Lemma_ord(w)=0}. Let us show that $N=\bigoplus_{\mathbb{Z}}\langle
a^{-n}\alpha a^n\rangle\cong \bigoplus_{\mathbb{Z}}\mathbb{Z}_3$. Suppose there is a
relation
\[
(a^{-n_1}\alpha^{\varepsilon_1} a^{n_1})(a^{-n_2}\alpha^{\varepsilon_2} a^{n_2})\ldots (a^{-n_k}\alpha^{\varepsilon_k} a^{n_k})=e,
\]
where $\varepsilon_i\in\{\pm 1\}$ and $n_1<n_2<\ldots<n_k$. Substituting $\alpha=ab^{-1}$
and $\alpha^{-1}=ac^{-1}$ in the previous expression, and making free cancelations we get a relation of the form
$a^{-n_1+1}w^{-1}a^{n_k}=e$, where $w$ is a word over $\{a,b,c\}$ with at least one
occurrence of $b$ or $c$. We get a contradiction with
Corollary~\ref{Cor_free_semigroup}, which proves our claim.

By Corollary~\ref{Cor_free_semigroup} the element $a$ has infinite order and
$N\cap\langle a\rangle=\{e\}$. Since $G_A=N\langle a\rangle$ and $a$ acts on $N$ by
conjugation via the shift, we get the statement of the theorem.
\end{proof}

\begin{remark} In the automaton realizations of
$\mathbb{Z}_m\wr\mathbb{Z}$ from \cite{GriZuk:lampl,SilvaStein:lamp} elements of the
subgroup $\bigoplus_{\mathbb{Z}}\mathbb{Z}_m$ are finitary transformations. An element $g$ is called finitary whenever there exists $n\in\mathbb{N}$ such that $g(uv)=g(u)v$ for all $u\in
X^n$ and $v\in X^{*}$, or equivalently, the sections of $g$ at all vertices of $n$-th level are
trivial. In our case, non-trivial elements of the subgroup
$N=\bigoplus_{\mathbb{Z}}\mathbb{Z}_3$ are not finitary.
\end{remark}

The stabilizers $St_{G}(w)$ of points $w\in X^{\mathbb{N}}$ are known as parabolic
subgroups of $G$. For every bireversible automaton, almost every point of
$X^{\mathbb{N}}$ with respect to the uniform measure on $X^{\mathbb{N}}$ has trivial
stabilizer (see \cite{SteinVor}). For our group $G_A$ there are points with a
non-trivial stabilizer, for example $a\in St_{G_{A}}(11\ldots)$. Moreover, using the
results obtained in \cite{gri_kravch} it is not difficult to prove that the stabilizer
$St_{G_A}(w)$ is non-trivial exactly when the sequence $w$ is eventually periodic. We
omit the proof because it uses quite different technique. Instead we just prove the
analog of Proposition~4.6 from \cite{SilvaStein:lamp}.

\begin{proposition}
The stabilizer $St_{G_A}(w)$ of every point $w\in X^{\mathbb{N}}$ is a cyclic group.
\end{proposition}
\begin{proof}
Note that a non-trivial element $g\in W\supset N$ has no fixed points in
$X^{\mathbb{N}}$, since $g$ corresponds to a non-trivial sequence of permutations
$(\pi_1, \pi_2, \ldots)\neq (\varepsilon,\varepsilon,\ldots)$, $\pi_i\in Alt_3$ and
$g(w)=g(x_1x_2\ldots)=\pi_1(x_1)\pi_2(x_2)\ldots$.

Let $g=na^k$, where $n\in N$ and $k\in\mathbb{N}$, be an element of $St_{G_A}(w)$ with
the smallest $k$. We prove that $g$ is a generator of $St_{G_A}(w)$. Take any $h\in
St_{G_A}(w)$, $h=n'a^m$, where $n'\in N$ and $m\in\mathbb{Z}$. Then $m$ is a multiple
of $k$; so let $m=kl$ with $l\in\mathbb{Z}$. We have $g^l= (na^k)^l=\overline{n}
a^{kl}$ for an opportune $\overline{n}\in N$. Then $hg^{-l}=n'\overline{n}^{-1}\in
N\cap St_{G_A}(w)=\{e\}$. Hence $h=g^l$.
\end{proof}

\section*{Acknowledgment}
This work was initiated while the first author was visiting Graz University of
Technology, whose support and hospitality are gratefully acknowledged. The second
author was supported by Austrian Science Fund project FWF P24028-N18. The third author
acknowledges support from the European Regional Development Fund through the programme
COMPETE and by the Portuguese Government through the FCT under the project
PEst-C/MAT/UI0144/2013 and the support of the FCT project SFRH/BPD/65428/2009. The
authors would also like to thank the developers of the program package
\textrm{AutomGrp} \cite{AutomGrp} which is been used to perform many of the
computations described in this paper.

\Addresses
\end{document}